\documentclass [12pt]{amsart}
\usepackage{hyperref}

\usepackage [margin=1.00in,marginparwidth=0.75in]{geometry} 

\usepackage{cite}
\usepackage{amssymb}
\usepackage{amsmath}
\usepackage{amsthm}
\usepackage{amsfonts,bbm}
\usepackage{graphicx}%
\usepackage{amscd}
\usepackage{color}
\usepackage[usenames,dvipsnames,svgnames,table]{xcolor}
\usepackage[toc,page]{appendix}
\usepackage{bbm}
\usepackage[normalem]{ulem}
\usepackage{soul,marginnote}
\usepackage{mathtools}

\newtheorem{thm}{Theorem}[section]
\newtheorem{cor}[thm]{Corollary}
\newtheorem{ex}[thm]{Example}

\newtheorem{lem}[thm]{Lemma}
\newtheorem{nota}[thm]{Notation}
\newtheorem{prop}[thm]{Proposition}

\newtheorem{as}[thm]{Assumption}

\theoremstyle{definition}

\newtheorem{df}[thm]{Definition}

\newtheorem{rem}[thm]{Remark}
\newtheorem{notation}[thm]{Notation}
\numberwithin{equation}{section}

\newcommand{\tX}{\ensuremath{{\widetilde{X}_x}}}

\newcommand{\tWa}{\ensuremath{{\widetilde{W}_{x,\alpha}}}}
\newcommand{\PSI}{\Psi}

\def\enddoc{\end{document}}

\global\newcount\NNN \global\NNN=1
\global\newcount\NNNg \global\NNNg=1

\newcommand{\rrem}[1]{{\color{purple} [*** #1 ***] \marginpar{\color{red}\mathversion{bold}$*{<} \the\NNN {>} *${\global\advance\NNN by 1}}}}
\newcommand{\grem}[1]{{\color{red} \ ***[#1]*** \ \marginpar{\color{OliveGreen}\mathversion{bold}$*{<} \the\NNNg {>} *${\global\advance\NNNg by 1}}}}

\newcommand{\rmr}[1]{{ \marginpar{ \color{red} \mathversion{bold}$*{<} \the\NNN {>} *$ #1 {\global\advance\NNN by 1}}}}
\newcommand{\rmg}[1]{{ \marginpar{ \color{OliveGreen} \mathversion{bold}$*{<} \the\NNN {>} *$ #1 {\global\advance\NNN by 1}}}}

\begin{document}

\title[Integral estimates of non-smooth Girsanov densities]
{
Ornstein-Uhlenbeck processes with singular drifts: integral estimates and Girsanov densities
}

\author[Gordina]{Maria Gordina{$^{*}$}}
\thanks{\footnotemark {$*$}Research was supported in part by NSF Grant DMS-1712427, and a Simons Fellowship.}
\address{Department of Mathematics\\
University of Connecticut\\
Storrs, CT 06269, U.S.A.} \email{maria.gordina{@}uconn.edu}

\author[R\"ockner]{Michael
R\"ockner{$^{\ddag}$}}
\thanks{\footnotemark {$^\ddag$} Research was supported in part by the German Science Foundation (DFG) through CRC 1283.}
\address{$^{\ddag}$ Department of Mathematics\\
Bielefeld University\\
D-33501 Bielefeld, Germany}
\email{roeckner{@}math.uni-bielefeld.de}

\author[Teplyaev]{Alexander Teplyaev{$^{\dag}$}}
\thanks{\footnotemark {$\dag$} Research was supported in part by NSF Grant DMS-1613025.}
\address{Department of Mathematics\\
University of Connecticut\\
Storrs, CT 06269, U.S.A.} \email{teplyaev{@}uconn.edu}

\keywords{Ornstein-Uhlenbeck process, singular perturbation, nonlinear infinite-dimensional stochastic differential equations, non-Lipschitz monotone coefficients, Girsanov theorem.}

\subjclass{Primary 60H10; Secondary 35R15, 60H15, 47D07, 47N30.}

\begin{abstract}
We consider a perturbation of a Hilbert space-valued Ornstein--Uhlenbeck process by a class of singular nonlinear non-autonomous maximal monotone time-dependent drifts.
The only further assumption on the drift is that it is bounded on balls in the Hilbert space uniformly in time. First we introduce a new notion of generalized solutions for such equations which we call pseudo-weak solutions and prove that they always exist and obtain pathwise estimates in terms of the data of the equation. Then we prove that their laws are absolutely continuous with respect to the law of the original Ornstein--Uhlenbeck process.
In particular, we show that pseudo-weak solutions always have continuous sample paths.
In addition, we obtain integrability estimates of the associated Girsanov densities.
Some of our results concern non-random equations as well, while probabilistic results are new even in finite-dimensional autonomous settings.
\end{abstract}

\maketitle
\tableofcontents

\renewcommand{\contentsname}{Table of Contents}


\section{Introduction}\label{s.1}

Suppose $H$ is a real separable Hilbert space with an inner product $\langle \cdot, \cdot \rangle$ and the corresponding norm $\vert \cdot \vert_{H}$. The aim of this paper is to study solutions to the following $H$-valued stochastic differential equation
\begin{equation}\label{e.2.1}
dX_{t}=\left( AX_{t}+
{F} \left( t, X_{t} \right)\right) dt+ \sigma dW_{t}, \ X_{0}=x \in H, t \geqslant 0.
\end{equation}
Here $W_{t}$ is a cylindrical Wiener process in $H$ on some filtered probability space $\left( \Omega, \mathcal{F}, \mathcal{F}_{t}, \mathbb{P} \right)$ satisfying the usual conditions of right continuity and $\mathbb{P}$-completeness. For the precise setting we refer to Section~\ref{ss.Assumptions}.

We consider Equation \eqref{e.2.1} without the standard assumption on $F (t, \cdot)$ being locally Lipschitz continuous. The motivation for our study includes a better understanding of equations such as \eqref{e.2.1} with time-dependent drifts of not necessarily polynomial growth.

Equation \eqref{e.2.1} can be viewed as a nonlinear non-autonomous perturbation of the stochastic differential equation corresponding to an Ornstein-Uhlenbeck process. In fact, it is a long-standing open problem to find optimal or nearly optimal conditions on $F $ such that \eqref{e.2.1} has a solution under the usual assumption that $A$ generates a $C_0$-semigroup on $H$ (see e.g.~\cite{DaPratoZabczykBook2014, DaPratoRockner2002} and the references therein).
If $F $ is maximal monotone and single-valued with $D_F=H$, then we can equivalently rewrite \eqref{e.2.1} as the random equation
\begin{equation}\label{eq:mr_1.2'}
 dZ_t = \left( A Z_t + F(t, Z_t + W_{0, A, \sigma} (t)) \right) dt, \quad Z_0 =x,
\end{equation}
where $Z_t=X_t-W_{0, A, \sigma}(t)$ and the Ornstein-Uhlenbeck process $W_{0,A,\sigma}$ solves \eqref{e.2.1} for $F_{}\equiv0$, $x =0$.
Moreover, in this case one can easily obtain a unique solution by classical results due to F.~Browder, Kato, Komura and Rockafellar in \cite{BrowderF1967a, Kato1967a, Komura1967, Rockafellar1970a}. However, the assumption $D_F = H$ excludes many interesting examples \cite{DaPratoRockner2002} and \cite[Section 7.2]{DaPratoZabczykBook2014}, and therefore we include the case $D_F \subsetneq H$, and also allow $F$ to be multi-valued, see Assumptions~\ref{A3} and~\ref{A4} below. The first main result of our paper is that under natural assumptions on $F$, namely, Assumptions~\ref{A3} and~\ref{A4}, \eqref{e.2.1} always has a solution in a generalized sense. Namely, we introduce \textbf{pseudo-weak solutions} in Definition~\ref{d-sol} and discuss them in detail in Section~\ref{A}.

The main results of the paper include a proof of existence of pseudo-weak solutions, pathwise \emph{a priori} estimates of these solutions in Section~\ref{s.3}, absolute continuity of the law of these solutions with respect to the law of the Ornstein-Uhlenbeck process, and finally integral estimates of the corresponding Radon-Nikodym densities in Section~\ref{s.4new}. Our approach can be interpreted as an extension of the classical use of Girsanov transformation to find a solution for a stochastic differential equation with a nonzero (but at most linearly growing) drift. The main idea behind results such as Theorem~\ref{t-sol-2} is that we can estimate $\varphi\left( \vert X_t \vert_{H}^{2}\right)$, where $X_t$ is a solution to Equation \ref{e.2.1} and $\varphi$ is a suitable function. The function $\varphi$ can be chosen by looking at the behavior of the nonlinearity $F $ at infinity, so the estimate only depends on the growth of $F $ at infinity, not on the nonlinearity itself.

Some of our results on the Girsanov transform are closely related to the infinite-dimensional estimates by D.~Gatarek and B.~Go{\l}dys in \cite{GatarekGoldys1992, GatarekGoldys1997}. They considered equations in Banach spaces, while we restrict our consideration to Hilbert space though for non-autonomous perturbations. In the future work we plan to extend our techniques to the reaction-diffusion equation in a Banach space. Our estimates of solutions and Girsanov densities are new even in finite dimensions, for example, compared to the ones due to N.~V.~Krylov in \cite{Krylov1990a, Krylov2002simple} and \cite[Chapter IV, \S 3]{KrylovDiffusionProcessesBook}.

We would like to comment on some of the previous results both in terms of the assumptions we make and the techniques we use. We describe the setting in Section~\ref{ss.Assumptions} in detail, including the assumptions on the coefficients of the non-autonomous equation \eqref{e.2.1}. The approach we use does not rely on an invariant measure, which is not available for non-autonomous equations, and therefore we do not use typical assumptions such as finite moments of the invariant measure and integrability properties of the nonlinear drift with respect to this measure.
The paper consists of three major parts which are intertwined: in Section~\ref{ss.PseudoWeakSolutions} we introduce a notion of pseudo-weak solutions to \eqref{e.2.1}, and prove their existence in Section~\ref{p-sol}. We use monotonicity of the coefficients of the equation to prove \emph{a priori} pathwise bounds in Theorem~\ref{t-sol-2}. In general one expects that our assumptions might imply weak
uniqueness, by appealing to Gronwall's lemma, but this seems out of reach for now in a general setting such as ours.
We refer to \cite{BogachevDaPratoRockner1996, DaPratoRockner2002} for a discussion of when and how martingale solutions to \eqref{e.2.1} can be constructed, and for more details on such solutions.

We also would like to mention several connections of our results to the cases when $F$ in \eqref{e.2.1} does not depend on $t$.
In this situation one can use our results to prove smoothness results for an invariant measure, closability of the corresponding Dirichlet form etc.
Note that in the current paper we do not address the question under which assumptions an invariant measure exists. Suppose there is an invariant measure as described in \cite{DaPratoRockner2002}, then one can use the Girsanov transform in Theorem \ref{t-sol-3} to show formally that the invariant measure is quasi-invariant under certain linear shifts. This leads to a possibility of using
\cite[Theorem 2.2]{AlbeverioRockner1990} and \cite[Theorem 1.3]{AlbeverioRockner1988a} to prove closability of the Dirichlet form. The main ingredient here is
the lower semicontinuity of the Radon-Nikodym density as described in \cite[p.122]{AlbeverioHoegh-Krohn1982} among other references. Finally, there are other recent approaches to quasi-invariance of semigroups in infinite dimensions which rely on functional inequalities
\cite{
	RocknerWang2013,
	GordinaRocknerWang2011,
	DriverGordina2009,
	DriverGordina2008,
	Gordina2017,
	Melcher2009,
	BaudoinGordinaMelcher2013,
	BaudoinFengGordina2019}, and these methods are not applicable to singular perturbations considered in the current paper.

\subsection*{Acknowledgements}
The authors thank G.~Da Prato, M.~Hairer, M.~Hinz and N.~Krylov for helpful discussions and suggestions.
The authors are grateful to anonymous referees for corrections and suggested improvements to the paper.

\section{Setting and main results}

\subsection{Setting and assumptions}\label{ss.Assumptions}

Let $H$ be a real separable Hilbert space with an inner product $\langle \cdot, \cdot \rangle$ and the corresponding norm $\vert \cdot \vert_{H}$. We denote by $B\left( H \right)$ the space of bounded linear operators equipped with the operator norm $\Vert \cdot \Vert$.
The Hilbert-Schmidt norm is denoted by $\Vert \cdot \Vert_{HS}$.
We suppose that the coefficients $A$, $F$ and $B$ in Equation \eqref{e.2.1} satisfy the following assumptions.

\begin{as}\label{A1} The operator $\left( A, D_{A} \right)$ generates a $C_{0}$-semigroup on $H$ denoted by $e^{tA}$, $t \geqslant 0$. We assume that there is $\beta > 0$ such that for all $x \in D_{A}$
\[
\langle Ax, x \rangle \leqslant -\beta \vert x \vert^{2}_{H}.
\]
\end{as}
Note that Assumption~\ref{A1} implies that $A$ is $m$-dissipative on $H$.

\begin{as}\label{A2} Both $\sigma$ and $\sigma^{-1}$ are in $B\left( H \right)$ with $\sigma$ being self-adjoint and positive such that for some $\varepsilon>0$
\[
\int_{0}^{T} t^{-\varepsilon}\Vert e^{tA} \sigma \Vert_{HS}^{2} \, dt < \infty, \text{ for all} T>0.
\]
\end{as}
Recall that under Assumption~\ref{A2} the Ornstein-Uhlenbeck process
\begin{equation}\label{eq:WxAsigma}
W_{x, A, \sigma} \left( t \right):= e^{tA}x+\int_0^t e^{(t-s)A} \sigma dW_s, \; t \geqslant 0,
\end{equation}
is pathwise continuous by \cite[Proposition 2.3]{DaPratoBook2004} which is based on the technique described in \cite{DaPratoKwapienZabczyk}.

\begin{as}\label{A3} Denote by $2^{H}$ the power set of the Hilbert space $H$. Let $F\left( t, \cdot \right): [ 0, \infty) \times D_{F} \to 2^{H}$ be a family of maps such that $D_{F}$ is a non-empty Borel set in $H$, and $dt\otimes \mathbb P$-almost surely the Ornstein-Uhlenbeck process $W_{x,A,\sigma} \in D_F$ for all $x \in D_F$. Furthermore, $F\left( t, \cdot \right)$ is an $m$-dissipative map, that is, for any $x_{1}, x_{2} \in D_{F}$
\[
\langle y_{1}-y_{2}, x_{1}-x_{2} \rangle \leqslant 0, \text{ for any } y_{1} \in F\left( t, x_{1} \right), y_{2} \in F\left( t, x_{2} \right), t \in [ 0, \infty)
\]
and for any $\alpha >0$ and $t \in [ 0, \infty)$
\[
\operatorname{Range}\left( \alpha I - F\left( t, \cdot \right) \right):=
\{ \alpha x - y:y\in F\left( t, x \right),x\in D_F \}=H.
\]
\end{as}

We refer to \cite[Section II.3]{BarbuBook1976} and \cite[Chapter 3]{BarbuBook2010} for basic facts about dissipative maps, as well as to the exposition in \cite{WiesingerPhDThesis}. In particular, it is known that in a Hilbert space a map is $m$-dissipative if and only if it is maximal dissipative, that is, it has no proper dissipative extensions. By \cite[Proposition 3.5(iv), Chapter II]{BarbuBook1976} for any $\left( t, x \right) \in [ 0, \infty) \times D_{F}$, the set $F\left( t, x \right)$ is non-empty, closed and convex, and so we can consider the well-defined
single-valued
map
\begin{equation}
F_{0}\left( t, x \right):=\left\{ y \in F\left( t, x \right): \vert y \vert_{H}=\inf \{ \vert z \vert_{H}, z \in F\left( t, x \right)\} \right\} \text{ for any } x \in D_{F}.
\end{equation}
This definition is the same as in \cite{DaPratoRockner2002} except that we allow dependence on time.
Using the Yosida approximation to $F$ described in Section~\ref{s2} we see that function $F_{0}\left( t, x \right)$, usually called \emph{the minimal section
of the maximal monotone operator} $F$, is Borel measurable.
Our next assumption is similar to the ones introduced in \cite{DaPratoZabczyk1992a, GatarekGoldys1997}.

\begin{as}\label{A4}
$F_0(t,\cdot)$ is uniformly bounded in $t\in[0,\infty)$ on balls in $H$, that is,
\begin{equation}\label{e.3.7}
 a\left( r \right):=\sup_{x \in B_r \cap D_F} \sup_{t\in[0,\infty)} \vert F_{0}\left( t, x \right) \vert_{H}
<\infty,
\end{equation}
where $B_{r}:=\left\{ x \in H: \vert x \vert_{H} <r \right\}$, $r>0$.
\end{as}
We define $a(0):=0$.
The function $a$ is non-decreasing and left-continuous,
and therefore Borel measurable.

We are mostly interested in the case when $\lim_{u \to \infty}a\left( u \right)=\infty$.
Assumption \eqref{A4} simply means that $F_{0}\left( t, x \right)$ is bounded on balls in its domain of definition in $x$, uniformly in $t$. In other words, we assume that function $F_{0}\left( \cdot, \cdot \right)$ is locally bounded in the space variable uniformly in time.

\subsection{Pseudo-weak solutions and their properties}\label{ss.PseudoWeakSolutions}

Throughout this paper we assume that Assumptions~\ref{A1}, \ref{A2}, \ref{A3} and~\ref{A4} hold. The first step in defining pseudo-weak solutions to Equation \eqref{e.2.1} requires suitable approximations to $F$. We use the Yosida approximation $F_{\alpha}, \alpha > 0$ described in Section~\ref{ss.Yosida} below.

By $Z_{\alpha, t}^{x}$ we denote the continuous $H$-valued process which is a mild solution to the family of regularized random ordinary differential equations
\begin{equation}\label{e.3.5}
dZ_{\alpha, t}^{x}=\left( AZ_{\alpha,t}+F_{\alpha}\left( t, Z_{\alpha,t}^{x}+W_{0, A, \sigma}\left( t \right)\right)\right) dt,
\qquad
Z_{\alpha,0}^{x}=x,
\end{equation}
where $W_{0, A, \sigma}$ is the pathwise continuous Ornstein-Uhlenbeck process defined by \eqref{eq:WxAsigma} with $x=0$. One can use \cite[Chapter 6, Theorem 1.2, page 184]{PazyBook1983} to justify the existence of mild solutions to \eqref{e.3.5}.
We note that technically speaking \cite{PazyBook1983} assumes that $F_\alpha$ is continuous in time, but it is clear that this assumption is not essential, and it is enough to assume joint measurability in time and space, and Lipschitz continuity in space, with the Lipschitz constant uniform in time, which holds for $F_{\alpha}$ as we explain in Section~\ref{ss.Yosida}.

The stochastic differential equation
\begin{align}
d X_{t} & =\left( AX_{t}+F_{\alpha}\left(t, X_{t} \right)\right)dt+ \sigma d W_{t}, \label{e.3.6}
\\
X_{0} &=x \in H \notag
\end{align}
has a mild solution $X_{\alpha, t}^{x}=X_{\alpha}\left( t, x \right), t \geqslant 0$, with $\mathbb P$-a.s.{\ }continuous sample paths.
Even though we have dependence on $\alpha$ in this equation, we prove that solutions $X_{\alpha, t}^{x}$ satisfy bounds \eqref{e.4.9}, which are uniform in $\alpha$.
In addition, it is clear that $Z_{\alpha,t}^{x}$ is a mild solution to the random ordinary differential equation \eqref{e.3.5} if and only if
\[
X_{\alpha, t}^{x}:=Z_{\alpha,t}^{x}+W_{0, A, \sigma}\left( t \right)
\]
is a mild solution to \eqref{e.3.6}.

Before proceeding to the notion of pseudo-weak solutions, we would like to comment on the intuition behind it. First we introduce pseudo-weak limits to deal with non-metrized topology. One of the consequences of this definition is that the convergence is governed by a function $\psi$, and therefore the limit might be different for different choices of $\psi$ as we discuss later. In what follows, unless stated otherwise, a pseudo-weak limit means a $\psi$-pseudo-weak limit, in the sense of Definition~\ref{d.pseudo-weak} and Remark~\ref{r.2.2} below.

\begin{df} \label{d-sol}
An adapted $H$-valued process $X^{x}$ is a \emph{pseudo-weak solution} to \eqref{e.2.1} if
it is a pseudo-weak limit point in
$L^2\left([0, \infty) \times \Omega, dt\times\mathbb{P}; H \right)$
of the approximating processes $X_{\alpha}^{x}$ defined by \eqref{e.3.6}.
\end{df}

\begin{rem}
Obviously, such pseudo-weak limit points are automatically adapted. More surprisingly, Theorem~\ref{t-sol-3} implies that they are also continuous $\mathbb{P}$-a.s. in $H$.
\end{rem}

The main results of our paper are summarized in the following three theorems. We start with pathwise \emph{a priori} estimates. For this purpose we introduce the function space $\mathcal{M}$ as follows.

\begin{df} \label{df-phi}
Let $\mathcal{M}$ denote the space of continuous functions $\varphi: [0, \infty) \to (0, \infty)$ such that
\begin{enumerate}
\item $\varphi$ is a strictly increasing convex function which is $C^2$ on $(0, \infty)$;
\item\label{nota-phi2} the limit $\frac{u\varphi^{\prime}\left( u \right)}{\varphi\left( u \right)}\xrightarrow[u \to \infty]{} L_{\varphi}$ exists, and $L_{\varphi} \in [1, \infty]$.
\end{enumerate}
\end{df}
For properties and examples of functions in $\mathcal{M}$ we refer to Section~\ref{s.Mspace}.

\begin{thm}[Uniform pathwise \emph{a priori} $\varphi$-estimates]\label{t-sol-2}
Under Assumptions~\ref{A1}, \ref{A2}, \ref{A3}, \ref{A4},
for every $\varphi\in\mathcal{M}$ we have the following estimates for a pseudo-weak solution $X_{t}^{x}$ to Equation \eqref{e.2.1}
\begin{equation}\label{e.4.9}
\varphi\left( \vert X_{t}^{x} \vert^{2}_{H}\right) \leqslant
 \frac{e^{-\beta t}}{2}\varphi\left(4 \vert x \vert^{2}_{H}\right) +\frac{1}{2}K_{\varphi}\left( t \right) +
 \frac{\beta t}{2} K_{\varphi, \beta, a}\left( t \right)< \infty
 \hskip0.1in a.s.
 \end{equation}
Here $K_{\varphi}\left( t \right)$ and $K_{\varphi, \beta, a}\left( t \right)$ are random functions defined in Notation~\ref{n.3.8} below. These functions only depend on $\beta, \sigma, A$ and $a$.
\end{thm}

\begin{thm}[Pseudo-weak solutions]\label{t-sol-1}
Under Assumptions~\ref{A1}, \ref{A2}, \ref{A3}, \ref{A4}, there exists a pseudo-weak solution $\left\{ X_t^x \right\}_{t \geqslant 0}$ to Equation \eqref{e.2.1}, i.e.

\begin{equation}\label{e-sol}
 X_t^x:=Z_{t}^{x}+W_{0, A, \sigma}\left( t \right) \hskip0.1in 
 \ t\geqslant0,
\end{equation}
where the process $Z_t^x$ is a pseudo-weak limit point of $Z_{\alpha, t}^{x}$, as $\alpha \to 0$, and $Z_{\alpha, t}^{x}$ is a solution to Equation \eqref{e.3.5}. Moreover,
$\left(dt\times\mathbb{P}\right)$-{a.s.} we have the following estimate
\begin{equation}\label{e|Z|}
 \vert X_{t}^{x} \vert_{H}
 \leqslant
 \vert x \vert_{H}
 e^{-\beta t}+\int_{0}^{t}e^{-\beta\left( t-s \right)}
 a\left(\vert W_{0, A, \sigma}\left( s \right)\vert_{H} \right) ds+|W_{0, A, \sigma}\left( t \right)|_{H}.
 \end{equation}
\end{thm}
In the next theorem we prove a Girsanov-type result with respect to the law of the
 Ornstein--Uhlenbeck process $W_{x, A, \sigma}$ defined by \eqref{eq:WxAsigma}.
 As we mentioned earlier, we can view this result as an analogue of using a Girsanov transformation to find a solution to stochastic differential equation where the reference process is the Ornstein--Uhlenbeck process.

\begin{thm}\label{t-sol-3}
Suppose Assumptions~\ref{A1}, \ref{A2}, \ref{A3}, \ref{A4} hold.
Then on any finite time interval $[0,T]$ and for any $x \in D_F$ the law of a pseudo-weak solution $X_t^x$ to Equation \eqref{e.2.1} is absolutely continuous with respect to the law of $W_{x, A, \sigma}$ on $L^2\left( [0,T], dt; H \right)$. In addition, the solution $X_t^x$ has $\mathbb{P}$-a.s continuous sample paths in $H$.
\end{thm}
\begin{rem}
One can expect that the corresponding Radon-Nikodym densities $\left\{ \rho^x \right\}_{x \in B_{r}}$ are uniformly integrable for any ball $B_{r}$, $r>0$.
Note that uniform integrability of the Radon-Nikodym densities $\left\{ \rho^x \right\}_{x \in B_{r}}$ in Theorem~\ref{t-sol-3} holds if and only if for any $r>0$ there exists an increasing unbounded function $\PSI:[0,\infty)\to[0,\infty)$ such that
$\sup_{x \in B_{r}}\mathbb{E}\rho^x\PSI\left(\rho^x\right) < \infty$.
In our paper we prove a weaker estimate
\begin{equation}\label{w-ent}
 \mathbb{E}\rho^x\PSI\left(\rho^x\right) < \infty.
\end{equation}
In Subsection~\ref{subsec-ex} we give quantitative estimates of $\PSI(\cdot)$ in terms of the function $a(\cdot)$ and the tail probability estimates of the Ornstein-Uhlenbeck process $W_{x, A, \sigma}$ under natural additional assumptions.

We stress that Theorem \ref{t-sol-3} is to the best of our knowledge new, even if $H=\mathbb R^d$.
\end{rem}

We prove Theorem~\ref{t-sol-2} and Theorem~\ref{t-sol-1} in Section~\ref{s.3}, where we provide more detailed statements as well. These results are illustrated by Examples~\ref{ex.3.4}, \ref{ex.3.5}, \ref{ex.3.6}. Theorem \ref{t-sol-3} is proved in Section \ref{s.4.3} and an example of how to construct $\PSI$ is given in Example~\ref{ex-psi}. Note that Theorem~\ref{t-sol-3} addresses the absolute continuity of the laws which is a long-standing question that has been implicitly stated in a number of publications such as \cite{Stannat1999a, Stannat1999b}.

\section{Preliminaries: Pseudo-weak convergence and Yosida approximations
}\label{s2}

\subsection{Pseudo-weak convergence}\label{A}
Let $\left( S, \mathcal{F}, \mu \right)$ be a $\sigma$-finite measure space, then for any $A \in \mathcal F$ we set
\[
\mu_A := \mathbbm{1}_A \mu.
\]
Let $H$ be a separable real Hilbert space with an inner product $\langle \cdot, \cdot \rangle$ and the corresponding norm $|\cdot|_{H}$.

\begin{notation} We need several spaces of $H$-valued functions on the measure space $\left( S, \mathcal{F}, \mu \right)$. By $L^{0}\left( S, \mu ; H \right)$ we denote the space of equivalence classes of $\mathcal{F}$-measurable function on $\left( S, \mathcal{F}, \mu \right)$ defined up to sets of $\mu$-measure zero and equipped with the topology of convergence in measure. By $L^{2}\left( S, \mu ; H \right)$ we denote the space of $H$-valued square-integrable functions on~$S$.
\end{notation}
In what follows we use $\xrightharpoonup[n \to \infty]{}$ for the weak convergence in Banach spaces.

\begin{df}\label{d.pseudo-weak} Suppose $f, f_{n} \in L^{0}\left( S, \mu; H \right)$, $n \in \mathbb{N}$. We say that $\left\{ f_{n}\right\}_{n=1}^{\infty}$ \emph{converges pseudo-weakly} to $f$, denoted by
\[
f_{n} \xrightharpoonup[n \to \infty]{\psi} f,
\]
if
\begin{equation}
\psi(f_n) \xrightharpoonup[n \to \infty]{} \psi(f) \text{ in } L^2\left( S, \mu_{A}; H \right) \text{ for any } A \in \mathcal{F} \text{ with } \mu(A)<\infty
\end{equation}
and for some $\psi: H \rightarrow H$ defined by
\begin{equation}\label{eq:mr3.1*}
\psi(h):=
\begin{cases}
\frac{h}{|h|_H} \, \psi_0(|h|_H), & \text{ if } h\not= 0, \\
0, & \text{ if } h=0,
\end{cases}
\end{equation}
where $\psi_0 : \mathbb R_+ \rightarrow \mathbb R_+$ is a strictly increasing continuous function such that $\psi_0(0)=0$. In this case we say that $f$ is a $\psi$-\emph{pseudo-weak limit}
of the sequence $\left\{ f_{n}\right\}_{n=1}^{\infty}$.
\end{df}

\begin{rem}\label{r.2.2} The definition of $\psi$-pseudo-weak limits depends on the choice of function $\psi$ which we usually fix. Typical examples for such a function $\psi$ are $\psi(h)=h$ or
\begin{equation}\label{e-psi0}
 \psi(h)=\frac{h}{1+|h|_H}, h \in H.
 \end{equation}
The latter choice corresponds to $\psi_0(r)=\dfrac{r}{1+r}$ for $r\in\mathbb R_+$.
For us the most interesting case is when $\psi$ is bounded, and most of the time we use $\psi$ defined by \eqref{e-psi0}, in which case we say \emph{pseudo-weak} convergence dropping the dependence on $\psi$.
\end{rem}
\begin{rem}\label{r.2.2-} The standard weak convergence coincides with $\psi$-pseudo-weak convergence for $\psi(h)=h$. For a space such as $H=L^2(\mathbb R^d)$, the converse can be proved as well: if the $\psi$-pseudo-weak convergence coincides with the usual weak convergence then $\psi(h)=h$ up to a multiplicative constant. We do not intend to study this question in detail in this paper, and only mention it in order to provide a better intuition for this notion of convergence.
\end{rem}

\begin{prop}\label{p.pseudo-weak}
Suppose $f, f_{n} \in L^0 \left( S, \mu ; H \right)$. Then for any bounded $\psi: H \longrightarrow H$ we have that $f_{n} \xrightharpoonup[n \to \infty]{\psi} f$
if and only	if
\begin{equation}
\int_A
	\left\langle \psi(f_{n}) -\psi(f), h \right\rangle_{}\,d\mu
	\xrightarrow[n \to \infty]{\hphantom{aaaaaaaa}}
	0
\end{equation}	
	for any $h\in H $ and any $A\in\mathcal F$ with $\mu(A)<\infty$.
\end{prop}

\begin{rem}\label{r.2.2.} Observe that for a fixed $\psi$ the pseudo-weak limit is \emph{unique}, that is, if
 \begin{align*}
 & f_{n} \xrightharpoonup[n \to \infty]{\psi} f,
 \\
 & f_{n} \xrightharpoonup[n \to \infty]{\psi} g,
 \end{align*}
 then $f=g$ $\mu$-a.e.
\end{rem}

\begin{rem}\label{r.2.2..} Note that, the topology of $L^{0}\left( S, \mu ; H \right)$ defined by convergence in measure
 \[
 \lim_{n\to\infty}\mu\left( \big\{ |f_n-f|_{H}>\varepsilon\big\} \cap A \right)=0 \quad \text{ for all} \varepsilon>0,\; A\in \mathcal F,\; \mu(A)< \infty
 \]
implies pseudo-weak convergence, but these two types of convergence are not equivalent in general.
\end{rem}

The following assertion is an easy consequence of the Banach-Saks Theorem applied to the Hilbert space $L^2\left( S, \mu ; H \right)$ or, more elementarily, of Fatou's Lemma.

\begin{prop}\label{p.A3} Suppose $f, f_n \in L^2\left( S, \mu ; H \right)$, $n \in \mathbb N$, and
\begin{equation}\label{e.weakAssump}
 f_{n} \xrightharpoonup[n \to \infty]{} f,
\end{equation}
then
\[
\vert f \vert_{H} \leqslant \limsup_{n \to \infty} \vert f_{n} \vert_{H} \hskip0.2in \mu\text{-a.e.}
\]
\end{prop}

\begin{cor}\label{corLimsup}
Let $f, f_n \in L^0\left( S, \mu; H \right)$, $n \in \mathbb N$, such that
\begin{equation*}
	f_{n} \xrightharpoonup[n \to \infty]{\psi} f.
\end{equation*}
Then
\begin{equation*}
|f|_H \leqslant \limsup\limits_{n \rightarrow \infty} |f_n|_H \quad \mu\text{-a.e.}
\end{equation*}
\end{cor}

\begin{proof}
Let $A \in \mathcal F$, $\mu(A) < \infty$ and
$\psi$ as in Definition~\ref{d.pseudo-weak}. Then by Proposition~\ref{p.A3} applied to $\mu_A$ instead of $\mu$ we have that on the set
\[
\left\{ f \neq 0, \limsup_{n \to \infty} \vert f_{n} \vert_{H} < \infty \right\}
\]
we have
\begin{align*}
& 0 < \psi_0(|f|_H)=|\psi(f)|_H \leqslant \limsup\limits_{n \rightarrow \infty} |\psi(f_n)|_H
\\
&
= \limsup\limits_{n \rightarrow \infty} \psi_0(|f_n|_H) \leqslant \psi_0\left(\limsup\limits_{n \rightarrow \infty} |f_n|_H \right) \hskip0.2in \mu_{A}-a.e.
\end{align*}
Applying the inverse of $\psi_0$ to both sides of this inequality and using the fact that $\mu$ is $\sigma$-finite proves the desired result.
\end{proof}

\begin{prop}\label{prop:mr3.5}
	Let $\psi $ in Definition \ref{d.pseudo-weak} be bounded.
 If $f_{n} \in L^{0}\left( S, \mu; H \right)$, $n \in \mathbb{N}$, are such that
\[
\sup_{n \in \mathbb{N}} \vert f_{n} \vert_{H} < \infty \hskip0.2in \mu-a.e.,
\]
then there exists $f \in L^{0}\left( S, \mu; H \right)$ such that for some subsequence $\left\{ n_{k} \right\}_{k \in \mathbb{N}}$
\[
f_{n_{k}} \xrightharpoonup[k \to \infty]{\psi} f.
\]
\end{prop}

\begin{proof}
Let $B_R(0)$ denote the open ball in $H$ with center $0$ and radius $0< R <\infty$.
Define $\psi^{-1}: B_{|\psi_0^{}|_\infty^{}} \longrightarrow H$ by
\begin{equation*}
\psi^{-1}(h):=
 \begin{cases}
 \frac{h}{|h|_H} \psi_0^{-1}(|h|_H), & \text{if} h\neq0,\\
 0, & \text{if} h=0,
 \end{cases}
\end{equation*}
where $\psi_0^{-1}$ is the inverse function of $\psi_0$. It is easy to see that $\psi^{-1}$ is indeed the inverse map of $\psi$ with $\psi$ defined by \eqref{eq:mr3.1*}. Now let $A\in \mathcal{F}$, $\mu(A)<\infty$, and
\begin{align*}
 V_n:=\psi(f_n), \quad n\in\mathbb{N}.
\end{align*}
Then $\left\{ V_n \right\}_{n\in\mathbb{N}}$ is bounded in $L^2\left( S, \mu_A; H \right)$. Hence there exists a $V_{A} \in L^2\left( S, \mu_A; H \right)$ such that for some subsequence $\left\{n_{k} \right\}_{k\in \mathbb{N}}$
\begin{align*}
V_{n_{k}}\xrightharpoonup[m\rightarrow\infty]{} V_A
\end{align*}
in $L^2\left( S,\mu_{A}; H \right)$.
Since $\mu$ is $\sigma$-finite, we can choose a sequence of subsets $A_l$ of finite $\mu$-measure such that $\cup A_l=S$, and by a diagonal argument we can construct $V \in L^2\left( S, \mu; H \right)$ and a subsequence, again indexed by $n_{k}, k\in \mathbb{N}$ such that
\begin{align*}
V_{n_{k}}\xrightharpoonup[k\rightarrow\infty]{} V.
\end{align*}
By Proposition~\ref{p.A3} we have that
\begin{align*}
|V|_H \leqslant \limsup_{k\rightarrow \infty} |V_{n_k}|_H = \limsup_{k\rightarrow \infty}\psi_0(|f_{n_k}|_H)
\leqslant \psi_0(\sup_{n\in\mathbb{N}}|f_n|_H) \ \mu\text{-a.e.}
\end{align*}
Therefore $V\in B_{|\psi_0^{}|_\infty^{}}(0)$ and
\begin{align*}
f:=\psi^{-1}(V)
\end{align*}
is well-defined. By definition of the $\psi$-pseudo-weak convergence
\begin{align*}
f_{n_{k}}\xrightharpoonup[ k\rightarrow\infty]{\psi}f.
\end{align*}
\end{proof}

\subsection{Yosida approximations to $A$}\label{ss.YosidaA}
We need the Yosida approximations $A_{\alpha}$ to $A$ for small $\alpha$, in particular, we will use the fact that such $A_{\alpha}$ satisfy Assumption~\ref{A1} with a change of $\beta$ as in Proposition~\ref{prop-beta-alpha}. Surprisingly, it is not easy to find a reference to this fact, so again we include it for completeness.

We start by recalling some standard facts about $C_{0}$-semigroups and their generators, most of this goes back to Hille and Yosida. We refer to \cite[Chapter II]{EngelNagelBook} for most of the material below. Let $\rho \left( A \right)$ be the resolvent set, then the resolvent of $A$ is defined as
\begin{align*}
& R_{\lambda} \left( A \right):= \left( \lambda I- A\right)^{-1}, \hskip0.1in \lambda \in \rho \left( A \right) \in B\left( H \right),
\\
& R_{\lambda} \left( A \right): H \longrightarrow D_{A}.
\end{align*}
Recall that for $\lambda >0$ we have $\Vert R_{\lambda} \left( A \right) \Vert \leqslant 1/\lambda$. In addition,
\begin{equation}\label{e.3.12}
\lambda R_{\lambda}\left( A \right)x \xrightarrow[\lambda \to \infty]{} x, \hskip0.1in x \in H.
\end{equation}
Note that $A R_{\lambda} \left( A \right)x= R_{\lambda} \left( A \right)Ax, \hskip0.1in x \in D_{A}$.
Finally, the Yosida approximations to $A$ are defined by
\begin{equation}
A_{\alpha}x:=\tfrac1\alpha A R_{\frac1\alpha} \left( A \right)x, x \in H.
\end{equation}
Since $(A,D_A)$ as a generator of a contractive $C_0$-semigroup is $m$-dissipative, $A_\alpha$ is a special case of $F_\alpha$ in \eqref{e.y.a}.
The Yosida approximations $A_\alpha$ to $A$ satisfy the following properties, see \cite[Proposition 7.2]{BrezisBook2011}, where
 \begin{equation}\label{eq-J-A}
 J_\alpha:=(I-\alpha A)^{-1},
\end{equation}
 $J_{\alpha} \in B\left( H \right)$, $\Vert J_\alpha\Vert \leqslant1$.
\begin{align}
& A_{\alpha}x \xrightarrow[\alpha \to 0]{} Ax, \hskip0.1in x \in D_{A}, \notag
\\
& | A_{\alpha}x|_H\leqslant|Ax|_H, \hskip0.1in x \in D_{A}, \label{e.3.10}
\\
& A_{\alpha}x=J_\alpha Ax, \hskip0.1in x \in D_{A}, \notag
\\
& A_{\alpha} \in B\left( H \right), \notag
\\
& \Vert A_{\alpha} \Vert \leqslant\tfrac1\alpha, \notag
\\
& A_{\alpha} =AJ_\alpha =\tfrac1\alpha (J_\alpha-I) . \notag
\end{align}

\begin{prop}\label{prop-beta-alpha}
Under Assumption \ref{A1} \[ \Vert J_\alpha \Vert \leqslant 1/({1+\alpha\beta}) \] and \[ \langle A_\alpha x, x \rangle \leqslant -\beta_\alpha \vert x \vert^{2}_{H}\]
for all $x\in H$, where
\[
\beta_\alpha:=\frac{\beta}{1+\alpha\beta}.
\]
\end{prop}
\begin{proof}
To prove the first inequality,
let $x\in H$ and $y:=J_\alpha x$, that is $x=y-\alpha A y$. Then
note that\[
|x|_H\cdot|y|_H\geqslant\langle x,y\rangle =\langle y-\alpha A y,y\rangle
\geqslant (1+\alpha\beta)|y|_H^2,
\]
which implies $|x| \geqslant (1+\alpha\beta)|y|$. To prove the second inequality, note that
\[
\langle -A_\alpha x, x \rangle = \langle -A y, y-\alpha A x\rangle
\geqslant
\beta|y|_H^2+\alpha|Ay|_H^2=\beta|y|_H^2+\frac1\alpha|x-y|_H^2 \geqslant \frac{\beta|x|_H^2}{1+\alpha\beta},
\]
where the last inequality is obtained by minimization over all $y\in H$.
\end{proof}
Note that the estimates in Proposition \ref{prop-beta-alpha} are best possible under Assumption \ref{A1}.

\subsection{Yosida approximations to $F$}\label{ss.Yosida}
Recall that to define pseudo-weak solutions in Definition~\ref{d-sol}, we used the Yosida approximation to $F$ satisfying Assumption~\ref{A3}. While there are standard references for this approximation such as \cite{BarbuBook1976, BarbuBook2010, BrezisBook1973}), and in the setting similar to the one considered in this paper in \cite{DaPratoRockner2002, DaPratoRocknerWang2009, WiesingerPhDThesis}, we include details for completeness: fix $t \in [0,\infty)$ and set $F :=F(t,\cdot)$. Then for any $\alpha > 0$ we define
\begin{equation}\label{e.y.a}
F_{\alpha}:=\frac{1}{\alpha}\left( J_{\alpha}\left( x \right) - x \right), x \in H,
\end{equation}
where
\[
J_{\alpha}\left( x \right):=\left( I -\alpha F\right)^{-1}\left( x \right), \ I\left( x \right)=x,
\]
which is a nonlinear generalization of $\eqref{eq-J-A}$.
Then each $F_{\alpha}$ is single-valued, dissipative, Lipschitz continuous with Lipschitz constant less than $\frac 2\alpha$ and satisfies
\begin{align}\label{e-F0}
& \lim_{\alpha \to 0} F_{\alpha}\left( x \right)=F_{0}\left( x \right), x \in D_F,
\\\label{e-F0-}
& \vert F_{\alpha}\left( x \right)\vert_{H} \leqslant \vert F_{0}\left( x \right)\vert_{H}, x \in D_F.
\end{align}
It is clear from the last inequality that for each $x_0 \in D_F$
\begin{equation}
|F_\alpha(t,x)|_{H} \le |F_0(t,x_0)|_{H}+ \frac 2\alpha|x|_{H} \le a(|x_0|_{H}) + \frac 2\alpha |x|_{H}, \quad x \in H.
\end{equation}

\section{Almost sure $\varphi$-type estimates of solutions $X_t$}
\label{s.3}

\subsection{Properties of function space $\mathcal{M}$} \label{s.Mspace}
Before proceeding to the proof of Theorem~\ref{t-sol-2} we need to establish properties of functions in $\mathcal{M}$ depending on the value of $L_{\varphi}$ as introduced in Definition \ref{df-phi}. In particular, we shall see that functions in $\mathcal{M}$ satisfy the standard condition in the de la Vall\'{e}e-Poussin Theorem. We also find sharp constants that might be useful for finding $\varphi$-moments depending on the growth of $F_0$ as measured by the radial function $a$ in Assumption~\ref{A4}.

\begin{lem}[]\label{l.3.1}
Suppose $\varphi \in \mathcal{M}$, then
\begin{itemize}
\item[(i)] for any $c>0$, $\beta >0$ and any $0<B<\beta L_{\varphi}$ there is a constant $C \geqslant 0$ such that
\[
\varphi\left( u \right)\left[ \frac{\varphi^{\prime}\left( u \right)}{\varphi\left( u \right)}\left( c\sqrt{u}-\beta u\right)+B \right] \leqslant C, \text{ for all} u \in (0, \infty ).
\]
The constant $C$ can be chosen as follows.
\begin{align}\label{e.3.8}
 & C \left( c, \beta, B \right):=\max\limits_{u\in \left[0, \infty \right)}\left( \varphi^{\prime}\left( u \right)\left( c\sqrt{u}-\beta u\right)+B\varphi\left( u \right) \right)=\\
 & \max\limits_{u\in \left[0, u_{0} \right]}\left( \varphi^{\prime}\left( u \right)\left( c\sqrt{u}-\beta u\right)+B\varphi\left( u \right) \right), \notag
\end{align}
where $u_{0}:=\max\left\{ \frac{c^{2}}{\beta^{2}}, \frac{c^{2}}{4\left(\beta - B \right)^{2}}\right\}$. In particular, for $B=\frac{\beta}{2}$
\[
C \left( c, \beta, \frac{\beta}{2} \right)=\frac{\beta}{2}\varphi\left( \frac{c^{2}}{\beta^{2}}\right).
\]

\item[(ii)] If $L_\varphi >1$, then
\end{itemize}
\[
\frac{\varphi\left( u \right)}{ u}\xrightarrow[u \to \infty]{} \infty.
\]
\end{lem}

\begin{proof}[Proof] (ii):
Define $H:(0, \infty) \longrightarrow (0, \infty)$ by $H\left( u \right):=\frac{\varphi\left( u \right)}{u}$. Then

\begin{align}\label{e.3.1}
& H^{\prime}\left( u \right)=\left( \frac{\varphi\left( u \right)}{u}\right)^{\prime}=
\left(\frac{u\varphi^{\prime}\left( u \right)}{\varphi\left( u \right)}-1\right)\cdot \frac{\varphi\left( u \right)}{u^{2}}
=
\left(\frac{u\varphi^{\prime}\left( u \right)}{\varphi\left( u \right)}-1\right)\cdot \frac{H\left( u \right)}{u}.
\end{align}
Using the assumption that $L_{\varphi}>1$ we see that there exists a $K>0$ such that
\[
H^{\prime}\left( u \right)= \left( \frac{\varphi\left( u \right)}{u}\right)^{\prime} > K \frac{H\left( u \right)}{u} >0
\]
for all large enough $u$. Thus
\[
\frac{H^{\prime}\left( u \right)}{H\left( u \right)}> \frac{K}{u}
\]
for such a $u$. Then for some $M>0$
\[
H\left( u \right)=\frac{\varphi\left( u \right)}{u}> Mu^K \text{ for all large enough} u,
\]
which implies that $\frac{\varphi\left( u \right)}{ u}\xrightarrow[u \to \infty]{} \infty$.

(i): It is enough to check that for $B \in (0, \beta L_\varphi)$
\[
\varphi\left( u \right)\left[ \frac{\varphi^{\prime}\left( u \right)}{\varphi\left( u \right)}\left( c\sqrt{u}-\beta u\right)+B \right]\xrightarrow[u \to \infty]{} -\infty,
\]
and so there is a $u_{0}>0$ such that
\[
\varphi^{\prime}\left( u \right)\left( c\sqrt{u}-\beta u\right)+B\varphi\left( u \right)<0 \text{ for all} u>u_{0}.
\]
Then we can choose
\begin{align}\label{eq:mr4.6'}
& C:= \max\limits_{u \in [0, u_{0}]}\left( \varphi^{\prime}\left( u \right)\left( c\sqrt{u}-\beta u\right)+B\varphi\left( u \right) \right) (>0).
\end{align}
Observe that
\[
\frac{\varphi^{\prime}\left( u \right)}{\varphi\left( u \right)}\left( c\sqrt{u}-\beta u\right)=\frac{u\varphi^{\prime}\left( u \right)}{\varphi\left( u \right)}\left( \frac{c}{\sqrt{u}}-\beta \right)
\xrightarrow[u \to \infty]{} -\beta L_{\varphi} \quad \big(:= -\infty, \text{ if} L_\varphi = \infty \big),
\]
and so
\[
\varphi\left( u \right)\left[ \frac{\varphi^{\prime}\left( u \right)}{\varphi\left( u \right)}\left( c\sqrt{u}-\beta u\right)+B \right] \xrightarrow[u \to \infty]{} -\infty.
\]
Recall that we can take $C$ to be the maximum of the following function
\[
f\left( u \right):=\varphi^{\prime}\left( u \right)\left( c\sqrt{u}-\beta u\right)+B\varphi\left( u \right).
\]
First we take the derivative of this function
\begin{align*}
& f^{\prime}\left( u \right)=\varphi^{\prime \prime}\left( u \right)\left( c\sqrt{u}-\beta u\right)+\varphi^{\prime}\left( u \right)\left( \frac{c}{2\sqrt{u}}-\beta \right)+B\varphi^{\prime}\left( u \right)=
\\
& \varphi^{\prime \prime}\left( u \right)\sqrt{u}\left( c-\beta \sqrt{u}\right)+\varphi^{\prime}\left( u \right)\left( \frac{c}{2\sqrt{u}}-\left(\beta-B\right) \right).
\end{align*}
By assumption $\varphi$ is an increasing convex function, and therefore $\varphi^{\prime \prime}$ and $\varphi^{\prime}$ are non-negative, so, since $\beta -B >0$, $f^{\prime}\left( u \right) \leqslant 0$ for any $u \geqslant u_{0}=\max\left\{ \frac{c^{2}}{\beta^{2}}, \frac{c^{2}}{4\left(\beta - B \right)^{2}}\right\}$. Therefore we can choose
\[
C\left( c, \beta, B \right)=\max\limits_{u\in [0, \infty )} f\left( u \right)=\max\limits_{u\in [0, u_{0}]} f\left( u \right).
\]
Finally, if $B=\beta/2$, then $u_{0}=\frac{c^{2}}{\beta^{2}}$, and $f^\prime(u)\ge 0$ on $[0,u_0]$, so
\[
C\left( c, \beta, \beta/2 \right)=f\left( u_{0} \right)=\frac{\beta}{2}\varphi\left( \frac{c^{2}}{\beta^{2}}\right).
\]
\end{proof}
We illustrate properties of functions in $\mathcal{M}$ by considering several fundamental examples.

\begin{ex}\label{ex.3.4}[Polynomial functions] Suppose $\varphi\left( u \right)=u^{p}, p \geqslant 1$, then $\varphi \in \mathcal{M}$. In this case $L_{\varphi}=p$. To see how we can find $C$ in \eqref{eq:mr4.6'}, observe that for any $0 < B < p \beta$

\begin{align*}
& f\left( u \right):=\varphi\left( u \right)\left[\frac{\varphi^{\prime}\left( u \right)}{\varphi\left( u \right)}\left( c\sqrt{u}-\beta u\right)+B\right]=
\\
& cpu^{p-1/2}+\left(B-p\beta\right)u^{p},
\end{align*}
for which

\begin{align*}
& f^{\prime}\left( u \right)=cp\left(p-\frac{1}{2}\right)u^{p-3/2}+\left(B-p\beta\right)pu^{p-1}=
\\
& p u^{p-3/2}\left( c\left(p-\frac{1}{2}\right)-\left(p\beta - B\right)\sqrt{u}\right).
\end{align*}
Then the maximum of $f$ is attained at $u_{0}=\left(\frac{c\left(p-\frac{1}{2}\right)}{p\beta - B}\right)^{2}$.
Therefore

\begin{align*}
& C_{p}\left( c, \beta, B \right):=\frac{c}{2}\left(\frac{c\left(p-\frac{1}{2}\right)}{p\beta - B}\right)^{2p-1}=\frac{c^{2p}}{2}\left(\frac{\left(p-\frac{1}{2}\right)}{p\beta - B}\right)^{2p-1}.
\end{align*}
In this example $L_{\varphi}=p$, and so by Lemma~\ref{l.3.1} for any $0 < B < \beta$ we can choose

\[
C_{1}\left( c, \beta, B \right):=\frac{c^{2}}{4\left( \beta - B \right)}.
\]
\end{ex}

\begin{ex}\label{ex.3.5}[Exponential] Suppose $\varphi\left( u \right)=e^{u}$, then $\varphi \in \mathcal{M}$. In this case $L_{\varphi}=\infty$, so we can take any positive constant $B$. For example, if $B=\beta/2$, then for

\[
f\left( u \right):=e^{u}\left[c\sqrt{u}-\beta u+B\right]=e^{u}\left[c\sqrt{u}-\beta u+ \frac{\beta}{2}\right]
\]
we have

\[
f^{\prime}\left( u \right)=e^{u}\left[c\sqrt{u} + \frac{c}{2\sqrt{u}}-\beta u - \frac{\beta}{2}\right]=e^{u}\left( \frac{c}{\sqrt{u}} - \beta \right)\left( u + \frac{1}{2} \right)
\]
and we can take

\[
C=f\left( \frac{c^{2}}{\beta^{2}} \right)= \frac{ \beta}{2}e^{\frac{c^{2}}{\beta^{2}}}.
\]
\end{ex}

\begin{ex}\label{ex.3.6} Suppose $\varphi\left( u \right)=u\ln \left( u + 1 \right)$, then $\varphi \in \mathcal{M}$. In this case $L_{\varphi}=1$, so we can take any $0 < B < \beta$ and then $C$ can be chosen by finding the maximum of the function

\begin{align*}
&f\left( u \right):=\varphi\left( u \right)\left[\frac{\varphi^{\prime}\left( u \right)}{\varphi\left( u \right)}\left( c\sqrt{u}-\beta u\right)+B\right]=
\\
& \ln\left( u+1 \right)\left( c\sqrt{u}-\left(\beta-B\right)u\right)+\frac{u}{u+1}\left( c\sqrt{u}-\beta u\right).
\end{align*}
Note that for $u>\left( \frac{c}{\beta-B} \right)^{2}$ the function $f\left( u \right)$ is negative. Therefore it is enough to find the maximum of $f$ on $\left( 0, \left( \frac{c}{\beta-B} \right)^{2} \right)$. We will use a rough estimate for $u \in \left( 0, \left( \frac{c}{\beta-B} \right)^{2} \right)$

\begin{align*}
&\ln\left( u+1 \right)\left( c\sqrt{u}-\left(\beta-B\right)u\right)+\frac{u}{u+1}\left( c\sqrt{u}-\beta u\right) \leqslant
\\
& \frac{c^{2}}{4\left(\beta-B\right)}\ln\left( u+1 \right)+\frac{c^{2}}{4\beta}\frac{u}{u+1}\leqslant
\\
& \frac{c^{2}}{4\left(\beta-B\right)}u+\frac{c^{2}}{4\beta} \leqslant \frac{c^{2}}{4\left(\beta-B\right)}\left( \frac{c}{\beta-B} \right)^{2}+\frac{c^{2}}{4\beta}.
\end{align*}
Thus we can take

\[
C\left( c, \beta, B \right):=\frac{c^{2}}{4}
\left(
\frac{c^{2}}{\left(\beta-B\right)^{3}}+\frac{1}{\beta}
\right).
\]
\end{ex}

\subsection{Proof of Theorem~\ref{t-sol-2}}
Recall now that the Ornstein-Uhlenbeck process $W_{x, A, \sigma}\left( t \right)$ defined by \eqref{eq:WxAsigma} is a Gaussian random variable with values in $H$ with mean $0$ and the covariance operator $Q_{t}$ given by
\[
Q_{t}x=\int_{0}^{t}e^{sA}\sigma^{2}e^{sA^{\ast}}x ds.
\]
We will use the following notation for the maximum process
\begin{equation}\label{e.3.9}
W_{x, A, \sigma}^{\ast}\left( t \right):=\sup_{s \in [0, t]} \vert W_{x, A, \sigma}\left( s \right)\vert_{H}.
\end{equation}

\begin{prop} For any $t>0$
there is an
$\varepsilon>0$ such that
\begin{equation}\label{e.3.2}
\mathbb{E} \left( e^{\varepsilon \, [W_{0, A, \sigma}^{\ast}\left( t \right)]^2} \right)<\infty.
\end{equation}
\end{prop}
\begin{proof}
This follows from Fernique's Theorem \cite{Fernique1975a}, see also
\cite[Theorem 3.1]{KuoLNM1975},
if one can show that the law of $W_{0, A, \sigma}$ is a Gaussian measure on $C([0,t];H)$.
In the case $A$ is self-adjoint, this follows from
\cite[Proposition I.0.7]{LiuRocknerBook2015}.
When $A$ is not necessarily self-adjoint
the same proof works using \cite[Theorem 2.9]{DaPratoBook2004}.	
\end{proof}

\begin{nota}\label{n.3.8} For any $\varphi \in \mathcal{M}$ and for all $t>0$ we denote the following random functions by
\begin{align*}
& K_{\varphi, \beta, a}\left( t \right):= \varphi\left( \frac{2\left[ a \left( W_{0, A, \sigma}^{\ast}\left( t \right)\right) \right]^{2}}{\beta^{2}}\right),
\\
& K_{\varphi}\left( t \right):=\varphi\left(2 \left| W_{0, A, \sigma} \left( t \right) \right|^{2}_{H} \right)
\end{align*}
Note that these functions are finite a.s.
\end{nota}

\begin{rem} We will make use of the following elementary inequalities: for any $a, b \geqslant 0$, and $p \geqslant 1$
 \begin{align}
 & \left(a + b\right)^{p} \leqslant 2^{p-1}\left(a^{p}+b^{p}\right), \label{e.2.7}
 \\
 & e^{\left( a+b\right)^{2}}\leqslant \frac{e^{4a^{2}}}{2}+\frac{e^{4b^{2}}}{2}.\notag
 \end{align}
\end{rem}
We are now in position to prove pathwise estimates in Theorem~\ref{t-sol-2}.
\begin{proof}[Proof of Theorem~\ref{t-sol-2}]
	\def\alambda{{\ensuremath\alpha^{\prime}}}
	\def\abeta{{\ensuremath\beta_\alambda}}
Suppose $\alambda>0$, $\alpha>0$ and $Z_{\alambda,\alpha, t}^{x}$ is a solution to
\begin{align}\label{mr4.10.2}
\begin{split}
dZ_{\alambda, \alpha, t}^x&=\left( A_\alambda Z_{\alambda, \alpha, t}^x+F_\alpha(t, Z_{\alambda, \alpha, t}^x + W_{0, A, \sigma}\left( t \right) \right)dt,\\
Z_{\alambda,\alpha,0}^x&=x.
\end{split}
\end{align}
Note that coefficients in \eqref{mr4.10.2} are Lipschitz, and therefore the solution exists and it is unique, and in addition the solution is continuous in $t$. Then for Lebesgue measure-a.e. $t>0$ and a $C^{1}$	function $\varphi: [0, \infty) \longrightarrow (0, \infty)$
\begin{align*}
\frac{d}{dt} &\varphi\left(\vert Z_{\alambda,\alpha,t}^{x} \vert^{2}_{H}\right)\\
&=2\varphi^{\prime}\left(\vert Z_{\alambda,\alpha,t}^{x} \vert^{2}_{H}\right)\langle \left( Z_{\alambda,\alpha,t}^{x} \right)^{\prime}, Z_{\alambda,\alpha,t}^{x} \rangle\\
&=2\varphi^{\prime}\left(\vert Z_{\alambda,\alpha,t}^{x} \vert^{2}_{H}\right)\langle {{{A_{\alambda}}}} Z_{\alambda,\alpha,t}^{x}+F_{\alpha}\left( t, Z_{\alambda,\alpha,t}^{x}+W_{0, A, \sigma}\left( t \right)\right), Z_{\alambda,\alpha,t}^{x} \rangle
\\
&= 2\varphi^{\prime}\left(\vert Z_{\alambda,\alpha,t}^{x} \vert^{2}_{H}\right)\langle {{{A_{\alambda}}}} Z_{\alambda,\alpha,t}^{x}+\left( F_{\alpha}\left( t, Z_{\alambda,\alpha,t}^{x}+W_{0, {{{A_{\alpha}}}}, \sigma}\left( t \right)\right)-F_{\alpha}\left( t, W_{0, A, \sigma}\left( t \right)\right)\right), Z_{\alambda,\alpha,t}^{x} \rangle
 \\
& \qquad + 2\varphi^{\prime}\left(\vert Z_{\alambda,\alpha,t}^{x} \vert^{2}_{H}\right)\langle F_{\alpha}\left( t, W_{0, A, \sigma}\left( t \right)\right), Z_{\alambda,\alpha,t}^{x} \rangle
\\
&\leqslant 2\varphi^{\prime}\left(\vert Z_{\alambda,\alpha,t}^{x} \vert^{2}_{H}\right)\langle {{{A_{\alambda}}}} Z_{\alambda,\alpha,t}^{x}, Z_{\alambda,\alpha,t}^{x} \rangle + 2\varphi^{\prime}\left(\vert Z_{\alambda,\alpha,t}^{x} \vert^{2}_{H}\right)\langle F_{\alpha}\left( t, W_{0, A, \sigma}\left( t \right)\right), Z_{\alambda,\alpha,t}^{x} \rangle\\
& \leqslant 2\varphi^{\prime}\left(\vert Z_{\alambda,\alpha,t}^{x} \vert^{2}_{H}\right) \left( -\abeta \vert Z_{\alambda,\alpha,t}^{x} \vert^{2}_{H}+ \vert Z_{\alambda,\alpha,t}^{x} \vert_{H} \, \vert F_{\alpha}\left( t, W_{0, A, \sigma}\left( t \right)\right) \vert_{H} \right)\\
& \leqslant 2\varphi^{\prime}\left(\vert Z_{\alambda,\alpha,t}^{x} \vert^{2}_{H}\right) \left(a \left( \vert W_{0, A, \sigma} \left( t \right) \vert_{H} \right) \vert Z_{\alambda,\alpha,t}^{x} \vert_{H} -\abeta \vert Z_{\alambda,\alpha,t}^{x} \vert^{2}_{H} \right),
\end{align*}
where we used Proposition~\ref{prop-beta-alpha}, Equation \eqref{e-F0-}, Assumption~\ref{A3} and Equation \eqref{e.3.7}. We mention that $F_\alpha$ is monotone, see for instance \cite[Appendix A]{CerraiLNM1762}.

Now suppose $\varphi \in \mathcal{M}$, then by Lemma~\ref{l.3.1} taking $B=\abeta/2$ and
\[
C:=C\left( a \left( |W_{0, A, \sigma}\left( t \right)|_{H} \right), \abeta, \frac{\abeta}{2} \right)=\frac{\abeta}{2}\varphi\left( \frac{\left[ a \left( |W_{0, A, \sigma}\left( t \right)|_{H} \right) \right]^{2}}{\abeta^{2}}\right),
\]
we obtain that for all $u \in (0, \infty)$
\[
\varphi^{\prime}\left( u \right)\left( a \left( |W_{0, A, \sigma}\left( t \right)|_{H} \right) \sqrt{u}-\abeta u\right) \leqslant C-\frac{\abeta}{2}\varphi\left( u \right),
\]
therefore
\[
\frac{d}{dt}\varphi\left(\vert Z_{\alambda, \alpha, t}^{x} \vert^{2}_{H}\right) \leqslant \abeta\left( \varphi\left( \frac{\left[ a \left( |W_{0, A, \sigma}\left( t \right)|_{H} \right) \right]^{2}}{\beta_\alambda^{2}}\right) - \varphi\left(\vert Z^{x}_{\alambda, \alpha, t} \vert^{2}_{H}\right)\right).
\]
Now by Gronwall's inequality we see that for all $t \geqslant 0$
\begin{equation}\label{e.Z.est}
\varphi\left(\vert Z_{\alambda, \alpha, t}^{x} \vert^{2}_{H}\right) \leqslant \varphi\left( \vert x \vert^{2}_{H}\right)e^{-\abeta t}+\abeta \int_{0}^{t}e^{-\abeta \left( t-s \right)}\varphi\left( \frac{\left[ a \left( |W_{0, A, \sigma}\left( s \right)|_{H} \right) \right]^{2}}{\beta_\alambda^{2}}\right) ds.
\end{equation}
Similarly to the proofs of \cite[Lemma 1.2.3, Lemma 1.3.1 and Appendix A]{CerraiLNM1762} one can show that $Z_{\alambda,\alpha,t}^x \longrightarrow Z_{\alpha,t}^x$ locally uniformly in $t\in[0,\infty)$ for $\mathbb{P}$-a.e. $\omega\in\Omega$. So, since $\varphi$ is continuous \eqref{e.Z.est} holds for $Z_{\alpha,t}^x$ replacing $Z_{\alambda,\alpha,t}^x$.

Now we can use \eqref{e.2.7} and the fact that $\varphi$ is convex to see that for the solution $X_{\alpha, t}^{x}$ to \eqref{e.3.6} we have

\begin{align*}
\varphi &\left( \vert X_{\alpha, t}^{x} \vert^{2}_{H}\right)
\\
&\leqslant \frac{1}{2}\varphi\left(2\vert Z_{\alpha, t}^{x} \vert^{2}_{H} \right)+\frac{1}{2}\varphi\left(2 \vert W_{0, A, \sigma}\left( t \right) \vert^{2}_{H} \right)
\\
&\leqslant \frac{1}{2}\varphi\left(2 \vert x \vert^{2}_{H} \right)e^{-\beta t}+ \frac{\beta}{2}\int_{0}^{t}e^{-\beta \left( t-s \right)}\varphi\left( \frac{2 \left[ a \left( |W_{0, A, \sigma}\left( s \right)|_{H} \right) \right]^{2}}{\beta^{2}}\right) ds +\frac{1}{2}\varphi\left(2 \vert W_{0, A, \sigma}\left( t \right) \vert^{2}_{H} \right)
\\
& \leqslant \frac{1}{2}\varphi\left(2 \vert x \vert^{2}_{H} \right)e^{-\beta t} +\frac{1}{2}\varphi\left(2 \vert W_{0, A, \sigma}\left( t \right) \vert^{2}_{H} \right) + \varphi\left( \frac{2\left[ a \left( W^{\ast}_{0, A, \sigma}\left( t \right) \right) \right]^{2}}{\beta^{2}}\right) \frac{\beta}{2}\int_{0}^{t}e^{-\beta \left( t-s \right)} ds
\\
& \leqslant \frac{e^{-\beta t}}{2}\varphi\left(2 \vert x \vert^{2}_{H} \right) +\frac{1}{2}\varphi\left(2 \vert W_{0, A, \sigma} \left( t \right) \vert^{2}_{H} \right) +
 \frac{\beta t}{2} \varphi\left( \frac{2 \left[ a \left( W^{\ast}_{0, A, \sigma}\left( t \right) \right) \right]^{2}}{\beta^{2}}\right).
\end{align*}
Here we replaced $\varphi\left( \cdot \right)$ by $\varphi\left( 2~\cdot \right)$ which is again in $\mathcal{M}$.

Thus we have an estimate for $\varphi \left( \vert X_{\alpha, t}^{x} \vert^{2}_{H}\right)$ which is uniform in $\alpha$, so we can apply $\varphi^{-1}$ to the above inequality and use Corollary~\ref{corLimsup} with $\left( S, \mu \right)=\left([0, \infty) \times \Omega, dt\times\mathbb{P} \right)$ to pass to the limit as $\alpha\rightarrow 0$ along a subsequence. Then we apply $\varphi$ to the resulting inequality to obtain \eqref{e.4.9}.
\end{proof}

\subsection{Further \emph{a priori} pathwise estimates of $X_t$ and proof of Theorem~\ref{t-sol-1}}\label{p-sol}

Below we prove more bounds on $X_{t}$ which in particular imply Theorem~\ref{t-sol-1}. Thus we work in the setting of Theorem~\ref{t-sol-1}, and in particular we assume that Assumptions~\ref{A1}--\ref{A4} hold.

\begin{prop}\label{p.4.Z}
Let $Z_{\alpha, t}^{x}$ be a solution to the regularized equation \eqref{e.3.5}. Suppose $Z_{t}^{x}$ is a pseudo weak limit point of $Z_{\alpha, t}^{x}, \alpha\to 0$.
Then almost surely for all $\alpha > 0$
\begin{equation}\label{e|Z|-}
\vert Z^x_{\alpha,t} \vert_{H}
 \leqslant
 \vert x \vert_{H}
 e^{-\beta t}+\frac{1}{2}\int_{0}^{t}e^{-\beta\left( t-s \right)}
 a\left(\vert W_{0, A, \sigma}\left( s \right)\vert_{H} \right) ds =: Z_t^{*,x}
\end{equation}
for all $t \geqslant 0$ and thus
\begin{align}\label{eq:mr5.2}
\vert Z^x \vert_{H} \leqslant Z^{*,x} \qquad (dt\times\mathbb{P})\text{-a.s.}
\end{align}
\end{prop}

\begin{proof}\def\alambda{{\ensuremath\alpha'}}
One of the observations in the proof of Theorem~\ref{t-sol-2} was that \eqref{e.Z.est} holds for $Z_{\alambda,\alpha,t}^x$ instead of $Z_{\alpha,t}^x$. Now we can take $\varphi$ to be the identity map and apply \cite[Theorem 5]{DragomirBook2003} to obtain \eqref{e|Z|-}. Equation \eqref{eq:mr5.2} then follows by Corollary~\ref{corLimsup} with $\left( S, \mu \right)=\left([0, \infty) \times \Omega, dt\times\mathbb{P} \right)$.
\end{proof}
\begin{proof}[Proof of Theorem~\ref{t-sol-1}]
The assertion follows from \eqref{e|Z|-} and Proposition~\ref{prop:mr3.5} with $\left( S, \mu \right)=\left([0, \infty) \times \Omega, dt\times\mathbb{P} \right)$.
\end{proof}

\section{Uniform integrability of Girsanov densities}\label{s.4new}

\subsection{Stopping times and Girsanov transforms}\label{subsec-Girsanov-stopping}
Fix $T>0$ and define a sequence of stopping times by
\begin{equation}\label{e-tau-n}
\tau_n^x:=\inf\{t \in[0,T] : Z_t^{*,x} +
|W_{0, A, \sigma}(t)|_{H} > n\} \land T, n\in\mathbb{N}.
\end{equation}
Note that stopping times $\tau_n^x$ do not depend on $\alpha$ and that $\mathbb{P}$-almost surely
\begin{equation}\label{e.4.7n}
\lim_{n\to\infty}\tau_n^x=T.
\end{equation}
Note also that $Z_t^{*,x} +
|W_{0, A, \sigma}(t)|_{H}$
as a $\mathbb P$-a.s.{\ }continuous process.

The following lemma is used in the proof of Theorem~\ref{t-sol-3}	in Section~\ref{s.4.3}.

\begin{lem}\label{l.4.4}
 We have that for any $n>3|x|_H$
\[ \mathbbm 1_{\{t\in[0,\tau^x_n]\}} |X_{t,\alpha}^{x}|_{H} \leqslant n\ \text{ for all } \alpha >0, t\geqslant0.
 \]
\end{lem}
\begin{proof} This follows immediately from \eqref{e|Z|-}.
\end{proof}
Now we consider Girsanov transforms for the Yosida regularized equations \eqref{e.3.5}. For $x\in H$ let
\begin{equation}\label{e-rho}
\rho_\alpha(x, t):=\exp (\zeta_\alpha(x, t)),
\end{equation}
where
\begin{align}
\zeta_\alpha(x, t)& :=\int\limits_{0}^{t}
\langle \sigma^{-1}F_\alpha(s,W_{x, A, \sigma}(s)), dW(s) \rangle \label{e-zeta}
\\
&
-\frac12\int_{0}^{t}|\sigma^{-1}F_\alpha(s,W_{x, A, \sigma}(s))|^2_{H} ds. \notag
\end{align}
We define the measure $\mathbb{P}^x_\alpha$ on $\left( \Omega, \mathcal{F}, \mathcal{F}_{t}, \mathbb{P} \right)$ by
\begin{equation}\label{e-P-x}
\frac{d\mathbb{P}^x_\alpha}{d\mathbb{P}}=\rho_\alpha(x,T)=:\rho_\alpha^x,
\end{equation}
and we denote by $\mathbb{E}^x_\alpha$ the expectation with respect to the probability measure $\mathbb{P}^x_\alpha$ given by \eqref{e-P-x}.
Note that this gives a
weak mild solution to \eqref{e.3.6} according to \cite[Appendix I]{LiuRocknerBook2015}.
More precisely, we define
\begin{equation}\label{e-tX}
\tX(t):=W_{x, A, \sigma}(t).
\end{equation}
Note that this process does not depend on $\alpha$ although the measure $\mathbb{P}^x_\alpha$ does depend on $\alpha$ which is important in \eqref{e-tX-} below. Denote
\begin{equation}\label{e-tW}
\widetilde{W}_{x,\alpha}\left( t \right):=W_t-\int_{0}^{t}\sigma^{-1}F_\alpha(s,W_{x, A, \sigma}(s))\,ds,
\end{equation}
then $\widetilde{W}_{x,\alpha}$ is a cylindrical Wiener process under $\mathbb{P}_\alpha^x$ and
\begin{align}
d\tX\left( t \right)&=d W_{x, A, \sigma} \left(t \right) = A\tX dt+ \sigma d W_t \label{e-tX-}
\\
&
= A\tX dt+F_\alpha(t,\tX(t))dt+\sigma d\tWa(t), \notag
\end{align}
that is, $\tX\left( t \right)$ solves this equation in the mild sense.

\begin{rem}[On localization]
As a side remark we would like to mention that in infinite dimensions the processes defined by \eqref{e-tX-} are not semimartingales in general (unlike in \cite{MetivierBook1982}). One might want to use localization to introduce
\begin{align}
& \tWa^n(t)=W_t-\int_{0}^{t\wedge\tau_n}\sigma^{-1}F_\alpha(s, W_{x, A, \sigma}(s))ds, \label{e-W-n}
\\
& \rho^n_\alpha(x, t)=\exp (\zeta_\alpha(x, t\wedge\tau_n^x)),\label{e-rho-n}
\end{align}
and then define $\rho_\alpha(x,t)$ as a limit as $n\to\infty$, if the limit exists.
However, the localization can not be used easily for the equations with non-smooth coefficients
because interchanging the limits as $n\to\infty$ and $\alpha\to 0$ may be problematic.
We use stopping times in a different way in \eqref{e.4.19-2}.
\end{rem}

\subsection{Estimates of the Girsanov densities and proof of Theorem~\ref{t-sol-3}}\label{s.4.3}

\begin{proof}[Proof of Theorem~\ref{t-sol-3}]
In this proof we assume that $x\in D_F$ and $T>0$ are fixed, and
$t\in[0, T]$. Subsequently we abuse notation, and drop dependence on $x, T$, although our estimates depend on $x$ and $T$.

By \eqref{e-P-x} we have that for any Borel-measurable $\PSI:[0,\infty) \rightarrow [0,\infty)$
\begin{equation}\label{eq:mr5.16}
\mathbb{E}\rho^{x}_\alpha\PSI\left(\rho^{x}_\alpha\right)
=\mathbb{E}_{\alpha}^x \PSI\left(\rho^{x}_\alpha\right),
\end{equation}
where $\rho^{x}_\alpha$ is the density defined by \eqref{e-P-x}.
Note that by \cite{Ondrejat2004a} the distribution of $(\tWa,\tX)$ under the measure $\mathbb{P}^x_\alpha$ is the same as the distribution of $(W, X^{x}_\alpha)$ under the measure $\mathbb{P}$ as we pointed out in Section~\ref{subsec-Girsanov-stopping}. In particular,
by Assumption~\ref{A3}
\begin{align}\label{mr5.15}
X_\alpha^x\in D_F \hskip0.2in dt\times\mathbb{P}-\text{a.s.}
\end{align}

Recall that $(\tWa,\tX)=(\tWa, W_{x, A, \sigma})$. Then
\begin{equation}\label{eq:mr5.17}
 \mathbb{E}^x_\alpha \PSI\left(\rho^{x}_\alpha\right) = \mathbb{E} \PSI\left({\widetilde{\rho^{x}_\alpha}}\right),
\end{equation}
where
\begin{equation}\label{enam}
{\widetilde{\rho^{x}_\alpha}}:=
\exp\left(
\int\limits_{0}^{t}
\langle \sigma^{-1}F_\alpha(s, X^{x}_\alpha(s)), dW(s) \rangle
+
\frac12\int\limits_{0}^{t}|\sigma^{-1}F_\alpha(s,X^{x}_\alpha(s))|_{H}^2\,ds
\right).
\end{equation}
We can estimate $\mathbb{E} \mathbbm{1}_A \vert \widetilde{\rho_\alpha^x}\vert^{p}$ for any $A\in\mathcal{F}$ as follows.
\begin{align*}
&\mathbb{E}\mathbbm{1}_A \left\vert \widetilde{\rho_\alpha^{x}} \right\vert^{p}=\\
 & \mathbb{E}\mathbbm{1}_A \exp\left( p\int_{0}^{t}\langle \sigma^{-1}F_{\alpha}(s, X^{x}_{\alpha}(s)), dW_{s}\rangle
 -p^{2}\int_{0}^{t}\vert \sigma^{-1}F_{\alpha}(s, X^{x}_\alpha(s)) \vert^{2}_{H} ds\right)\\
 &\times\exp\left( \left( p^{2}+\frac{p}{2}\right)\int_{0}^{t}\vert \sigma^{-1}F_\alpha(s, X^{x}_\alpha(s)) \vert^{2}_{H} ds\right)\\
 &\leqslant\left( \mathbb{E}\exp\left( 2p\int_{0}^{t}\langle \sigma^{-1}F_\alpha(s, X^{x}_\alpha(s)), dW_{s}\rangle
 -2p^{2}\int_{0}^{t}\vert \sigma^{-1}F_\alpha(s, X^{x}_\alpha(s)) \vert^{2}_{H} ds\right)\right)^{1/2}\\
&\times\left(\mathbb{E}\mathbbm{1}_A \exp\left( \left( 2p^{2}+p\right)\int_{0}^{t}\vert \sigma^{-1}F_\alpha(s, X^{x}_\alpha(s)) \vert^{2}_{H} ds\right)\right)^{1/2}.
\end{align*}
Note that the first term in the last equation is equal to the expectation of the stochastic exponential for the martingale
$2p\int_{0}^{t}\langle \sigma^{-1}F_\alpha(s, X^{x}_\alpha(s)), dW_{s}\rangle$, and so its expectation is $1$.
Therefore,
\begin{equation}\label{e.4.19}
\mathbb{E}\mathbbm{1}_A \left\vert\widetilde{\rho_\alpha^{x}} \right\vert^{p} \leqslant \left(\mathbb{E}\mathbbm{1}_A \exp\left( \left( 2p^{2}+p\right)\int_{0}^{t}\vert \sigma^{-1}F_\alpha(s, X^{x}_\alpha(s)) \vert^{2}_{H} ds\right)\right)^{1/2}.
\end{equation}

Let us first assume that $\sup\limits_{r>0}a(r)<\infty$. Then taking
$A:=\{\tau_n^x\geqslant t\}$, $n>3|x|_H$, it follows by \eqref{mr5.15}, \eqref{e.3.10} and Assumption~\ref{A4} that
the right-hand side in \eqref{e.4.19} is bounded uniformly in $\alpha$, and therefore
we can choose $\Psi(y)=y^p$
for any $p>0$ to conclude by \eqref{eq:mr5.16} that $\rho_\alpha^x$, $\alpha>0$, are uniformly $L^1(\Omega,\mathbb P)$-integrable.
Now let us assume that
 $$\sup\limits_{r>0}a(r)=\infty.$$ Then we
can construct $\Psi$ as follows.
We take $p=2$ and use Assumption~\ref{A4}, Equations
\eqref{eq:mr5.16}, \eqref{mr5.15}, \eqref{eq:mr5.17}, Proposition~\ref{p.4.Z} and Lemma~\ref{l.4.4},
to deduce that
\begin{equation}
\mathbb{E} \left\vert\widetilde{\rho_\alpha^{x}} \right\vert^{2}
\mathbbm{1}_{ \left\{ \tau_n^x \geqslant t \right\}}
\leqslant \exp\left( 5 \left( a(n) \|\sigma^{-1}\|
\right)^{2} T\right).
\end{equation}
By \eqref{e.4.19} and Chebyshev's inequality we have
\begin{equation*}
\mathbb{P} \left({\widetilde{\rho_\alpha^{x}}} > y \right)\leqslant
\mathbb{P} \left({\widetilde{\rho_\alpha^{x}}} > y,\tau_n^x \geqslant t \right)+\mathbb{P} \left(\tau_n^x < t \right)
\end{equation*}
\begin{equation}\label{e.4.19-2}
\leqslant
\exp\left( 5 \left( a(n) \| \sigma^{-1} \|
\right)^{2} T\right)/y^2+\mathbb{P} \left(\tau_n^x < T \right).
\end{equation}
For $y\in \big( 1+\exp (5(a(0) \| \sigma^{-1} \| )^2T),\infty \big)$ we define $n(y)$ as the maximal natural number such that
\[
\exp\left( 5\left( a(n(y)) \| \sigma^{-1} \|
\right)^{2} T\right)< y
\]
or
\begin{equation}\label{e-a-n}
a(n(y))
<\frac1{\| \sigma^{-1} \|}\left( \frac{\log(y)}{5T}\right)^{1/2}.
\end{equation}
Recall that we assumed in this part of the proof that $a(n)$ is non-decreasing and unbounded, and therefore we have that $n(y)$ is non-decreasing in $y$ and
$
\lim_{y\to\infty}n(y)=\infty
$.
Then we can define the non-increasing function by
\begin{equation}\label{e.4.23}
p_0 (y):= p_0 (y, \sigma, A, a, x, T):=\min \left\{ 1, \, 1/y+\mathbb{P} \left(\tau_{n(y)}^x < T \right) \right\}, \ y\in(0,\infty).
\end{equation}
Observe that by \eqref{e.4.7n}
\[
\lim_{y\to\infty}p_0 (y, \sigma, A, a, x, T)=0.
\]
Suppose $\PSI:[0,\infty)\rightarrow[0,\infty)$ is an increasing function which is $C^1$ on $(0,\infty)$ with $\PSI(0)=0$. Then obviously
\begin{equation}
 \mathbb{E} \PSI\left( Y \right)
 = \mathbb{E}\int_{0}^{\infty}\PSI^{\prime}(y)\mathbbm{1}_{\{0<y<Y\}} dy
 = \int_{0}^{\infty}\PSI^{\prime}(y)\mathbb{P}\{ Y >y\}dy
\end{equation}
for any non-negative random variable $Y$. We apply this equation to $Y=\widetilde{\rho}^x_\alpha$ to see that by \eqref{e.4.19-2}, \eqref{e.4.23} for all $\alpha>0$
\begin{align}\label{e-Psi'}
\mathbb{E} \PSI(\widetilde{\rho}_{\alpha}^x)\leqslant \int_0^{\infty} \PSI^\prime (y)p_0 (y)dy.
\end{align}
Thus by \eqref{eq:mr5.16} and \eqref{eq:mr5.17} this implies uniform $L^1(\Omega,\mathbb{P})$-integrability of $\rho_\alpha^x$, $\alpha>0$, if we can find $\PSI$ as above with the following two properties
\begin{align}\label{eq:mr5.26}
\int_0^\infty \PSI^\prime(y)p_0 (y)dy<\infty
\end{align}
and
\begin{align}\label{eq:mr5.27}
\lim_{y\rightarrow\infty}\PSI(y)=\infty.
\end{align}
The existence of such a $\PSI$ can be seen as follows: since $y\mapsto p_0 (y)$ decreases to zero as $y\rightarrow\infty$, we can find a sequence $\left\{ y_k \right\}_{k\in\mathbb{N}}$ in $(0,\infty)$ such that
\begin{align*}
y_k+3<y_{k+1},\ k\in\mathbb{N},
\end{align*}
and $p_0 (y)\leqslant\frac{1}{k^2}$ for $y\geqslant y_k$. Now define $g:[0,\infty)\rightarrow[0,\infty)$ by
\begin{align*}
g=\sum_{k=1}^\infty \PSI_k,
\end{align*}
where $\PSI_k \in C^\infty((0,\infty))$ such that
\begin{align*}
\mathbbm{1}_{[y_k+1,y_k+2]}\leqslant\PSI_k\leqslant \mathbbm{1}_{[y_k,y_k+3]}.
\end{align*}
Define
\begin{align*}
\PSI(y)=\int_0^y g(s)ds,\ y \geqslant 0.
\end{align*}
Then $\PSI:[0,\infty)\rightarrow [0,\infty)$ is continuous, increasing and $C^\infty$ on $(0,\infty)$. Furthermore, obviously \eqref{eq:mr5.26} and \eqref{eq:mr5.27} hold. Another construction of a function $\PSI$ is given in the beginning of Subsection~\ref{subsec-ex}.

Now let $X^x\in L^2\left(\Omega, \mathbb{P}; L^2\left( [0,T], dt; H \right) \right)=L^2\left([0,T] \times \Omega, dt \times \mathbb{P}; H \right)$ be the pseudo-weak limit of $X^x_{\alpha_n}$ as $\alpha_n\xrightarrow[n\rightarrow\infty]{} 0$, with the corresponding function $\psi: H\rightarrow H$ defined by \eqref{eq:mr3.1*}. By the above and the Dunford-Pettis theorem, choosing another subsequence if necessary we have
\begin{align*}
\rho^x_{\alpha_n}\xrightharpoonup[n\rightarrow \infty]{}\rho^x\ \text{in}\ L^1(\Omega, \mathbb{P})
\end{align*}
for some $\rho^x\in L^1(\Omega, \mathbb{P})$.

Let $\mathcal{X}:=L^2([0,T];H)$ and let $G:\mathcal{X}\rightarrow\mathbb{R}$ be bounded and sequentially weakly continuous. Then for $\mathbb{Q}_x:=\mathbb{P}\circ W_{x,A,\sigma}^{-1}$

\begin{align}\label{eq:Grho}
\begin{split}
\int_{\mathcal{X}} G\circ\psi \ d\big(\mathbb{P}\circ (X^x)^{-1}\big)
&=\int_\Omega G(\psi (X^x)) \ d\mathbb{P}
=\lim_{n\rightarrow\infty}\int_\Omega G(\psi (X^x_{\alpha_n}))\ d\mathbb{P}\\
&=\lim_{n\rightarrow\infty}\int_\Omega (G\circ \psi )(W_{x,A,\sigma})\rho_{\alpha_n}^x d\mathbb{P}
=\int_\Omega (G\circ\psi )(W_{x,A,\sigma})\rho^x\ d\mathbb{P}\\
&=\int_\Omega (G\circ\psi )(W_{x,A,\sigma}) \, \mathbb{E}_{\mathbb{P}}[\rho^x \,| \, W_{x,A,\sigma}]\ d\mathbb{P}
=\int_{\mathcal{X}} G\circ\psi \ \overline{\rho}^x \ d\mathbb{Q}_x,
\end{split}
\end{align}
where
\begin{align*}
\overline\rho^x:=\mathbb{E}_\mathbb{P}[\rho^x \, | \, W_{x,A,\sigma}=\cdot].
\end{align*}

Let $\mathcal{N}$ denote the set of all such functions $G\circ\psi: H \longrightarrow \mathbb{R}$ from above. Since $\psi$ in \eqref{eq:mr3.1*} is one-to-one, we can find a countable set $\mathcal{N}_0\subset \mathcal{N}$, which separates the points in $\mathcal{X}$. Indeed, let $\{e_i \}_{i \in \mathbb{N}}$ and $\{g_i\}_{i \in \mathbb{N}}$ be orthonormal bases of $H$ and $L^2\left( [0,T], dt;\mathbb{R} \right)$ respectively. Define maps $G_{ij}:\mathcal{X}\rightarrow\mathbb{R}$, $i, j\in\mathbb{N}$,

\begin{align*}
G_{ij}(w):=\int_0^T g_i(t)\langle e_j,w(t)\rangle_Hdt,\ w\in\mathcal{X}.
\end{align*}
Then obviously $\{G_{ij}, i,j\in\mathbb{N}\}$ separates the points of $\mathcal{X}$ and hence so does $\mathcal{N}_0:=\{(N\land G_{ij}\lor (-N))\circ\psi, i,j,N \in \mathbb{N}\}$. Clearly, each $N\land G_{ij}\lor (-N)$ is weakly continuous, so $\mathcal{N}_0\subset\mathcal{N}$. Furthermore, obviously $\mathcal{N}$ is closed under multiplication and consists of bounded Borel measurable functions on $\mathcal{X}$. Therefore, \eqref{eq:Grho} implies that
\begin{align*}
\mathbb{P}\circ (X^x)^{-1}=\overline{\rho}^x \mathbb{Q}_x.
\end{align*}
Note that by Kuratowski's Theorem (e.g. \cite[Section I.3]{ParthasarathyBook2005}) $C([0,T];H)$ is a Borel subset of $\mathcal{X}$ such that $\mathbb{Q}_x(C([0,T];H))=1$, therefore
\begin{align*}
(\mathbb{P}\circ (X^x)^{-1})(C([0,T];H))=1,
\end{align*}
so $X^x$ has continuous sample paths $\mathbb{P}$-a.s.
\end{proof}

\subsection{Quantitative estimates
and examples}\label{subsec-ex}

As in Subsection~\ref{s.4.3} we fix $x\in D_F$ and $T>0$ below, and
omit dependence on $x,T$ in the notation.

One of the consequences of the proof of Equation \eqref{w-ent} in Theorem~\ref{t-sol-3} is a constructive approach to finding a function $\PSI(\cdot)$, though it seems not possible to find an optimal form of the function $\PSI(\cdot)$ under conditions of Theorem~\ref{t-sol-3}. Below we present some explicit results. In particular, we explain how to treat the case when the growth of the nonlinearity given by function $a \left( \cdot \right)$ is polynomial.

We begin with the following simple observations about real-valued functions. Assume that we are given a non-increasing function
$
p_0:[0,\infty) \longrightarrow (0,\infty)
$ which
is continuous
at zero and such that
\begin{equation}\label{e-p0-}
\lim\limits_{y\to\infty}p_0 (y)=0.
\end{equation}
Our aim is to find a non-decreasing absolutely continuous function $\Psi:[0,\infty) \longrightarrow [0,\infty)$
 such that
\eqref{eq:mr5.26} and \eqref {eq:mr5.27} hold.
If $p_0 (\cdot)$ is absolutely continuous, then we can simply choose   an  absolutely continuous $\Psi:[0,\infty) \longrightarrow [0,\infty)$ satisfying
\begin{equation}\label{e-PSI-p}\Psi(y)\leqslant  (p_0 (y))^{-1/2}\end{equation}
and show that
\begin{equation}\label{e-psi'}
\int_0^\infty\Psi^{\prime}(y)p_0 (y)dy
=
 -\Psi(0)p_0(0)-\int_0^\infty\Psi(y)p_0^{\prime} (y)dy\leqslant
2\sqrt{p_0(0)}-\Psi(0)p_0(0).
\end{equation}
Even if $p_0 (\cdot)$ is not absolutely continuous,   we can define for $y\in\mathbb R$ and $\delta\in(0,\infty)$
\begin{equation}\label{e-mollifier}
p_\delta(y):=\frac1\delta\,\int_{\mathbb R} \,
p_0 (s-\delta)\,m\left(\frac{s-y}\delta\right)
ds\geqslant p_0 (y),
\end{equation}
where $m(\cdot)$ is a standard mollifier on $\mathbb R$, that is, a smooth non-negative function $m:\mathbb R \to [0,\infty)$ with support in $[-1,1]$ and $L^{1}$-norm equal to $1$. We assume that in this formula $p_0 (y)=p_0 (0)$ if $y\leqslant0$. Equation \eqref{e-mollifier} is different from the usual mollification because of the shift by $\delta$ in $p_0 (s-\delta)$, which ensures that
$p_{\delta}(y)$ is a non-decreasing function of $\delta$. Then $p_{\delta}(y)$ satisfies \eqref{e-p0-}, and this implies the existence of  a smooth function $\PSI(y)=(p_\delta (y))^{-1/2}$ satisfying \eqref{e-PSI-p}.

\begin{rem} For any fixed $\sigma, A, a, x, T$ one can estimate $\PSI(y)$ in \eqref{w-ent}. However, in the most general form this computation is cumbersome and therefore it is not presented in our paper. Instead we illustrate this approach by giving several examples satisfying natural extra assumptions.
\end{rem}
The following is a corollary from the proof of Theorem~\ref{t-sol-3} given in Subsection~\ref{s.4.3}.

\begin{ex}\label{ex-psi}
Suppose \begin{equation}\label{e-yn}
a(y) \leqslant c(1+y)^{n+1}
\end{equation} for some $c\geqslant1, n>0$ and all $y\geqslant0$. Then
we can choose $ \delta_0>0$
such that
\begin{equation}\label{e-PSI-d}
\PSI(y):=\exp\left( \big(\log(1+y) \big)^{\delta_0} \right)
\end{equation}
 satisfies \eqref{w-ent}.

\begin{proof}
Our aim is to show that for $\delta_0>0$ small enough, if we choose $\PSI(y)$ as in
\eqref{e-PSI-d}, it satisfies \eqref{e-PSI-p} and \eqref{e-psi'}.
	
We define a smooth strictly increasing function $a_0:[0,\infty)\to[0,\infty)$ by
\begin{equation}\label{eq-a0y}
a_0(y):=1+y+\int_{-1}^\infty \,
a (s+1)\,m\left({s-y}\right)
ds> a (y)+y,
\end{equation}
where $m(\cdot)$ is a standard mollifier on $\mathbb R$ with support in $[-1,1]$.

Then for $y\in \big( \exp (5(a_0(0) \| \sigma^{-1} \| )^2T),\infty \big)$ we define
\begin{equation}\label{eq-ny}
n(y):=
\left[a_0^{-1}\left( \frac1{\| \sigma^{-1} \|}\left( \frac{\log(y)}{5T}\right)^{1/2} \right)\right],
\end{equation}
where $[z]$ is the integer part of $z$ and $a_0^{-1}(\cdot)$ is the inverse function of the function $a_0(\cdot)$. Note that
\begin{equation}\label{e-a-n-}
\exp\left( 5\left( a(n(y)) \| \sigma^{-1} \|
\right)^{2} T\right)
<
\exp\left( 5\left( a_0(n(y)) \| \sigma^{-1} \|
\right)^{2} T\right)\leqslant
y
\end{equation}
and
\begin{equation}\label{e-a-n-u}
\exp\left( 5\left( a_0(n(y)+1) \| \sigma^{-1} \|
\right)^{2} T\right)> y
\end{equation}
for all $y\in \big( \exp (5(a_0(0) \| \sigma^{-1} \| )^2T),\infty \big)$.
Equation \eqref{e-yn} implies that
\[a_0(y) \leqslant c_0(1+y)^{n+1}\]
for some $c_0\geqslant c$ and all $y\in \big( 1+\exp (5(a(0) \| \sigma^{-1} \| )^2T),\infty \big)$.

Recall that $ W_{0, A, \sigma}^{\ast}=\sup_{t\in[0,T]} |W_{0, A, \sigma}(t)|_H$.
Applying Markov's inequality to Fernique's Theorem as formulated in \eqref{e.3.2},
we have that
\[\mathbb{P}\left( W_{0, A, \sigma}^{\ast} > s\right)<\exp \left(-\varepsilon s^2\right)\mathbb{E} \left( e^{\varepsilon \, [W_{0, A, \sigma}^{\ast}\left( T \right)]^2} \right) \]
 for all $s>0$.
If $p_0 (y)$ is given by \eqref{e.4.23} then, by \eqref{e-PSI-p} and \eqref{e-psi'}, it is enough to verify the following inequality
 \[
 p_0 (y)
 \leqslant
 2\mathbb{P}\left( W_{0, A, \sigma}^{\ast} \geqslant C_1 \,a_0^{-1}\left( \,a_0^{-1}\left( \sqrt{ \log(2+y)}\right)\right)\right)
 \]
 for some constant $C_1 >0$ and all $y\in(y_0,\infty)$ for some $y_0>0$.
 It is enough to show that
 \[
 \mathbb{P} \left(\tau_{n(y)}^x < T \right)
 \leqslant
 \mathbb{P}\left( W_{0, A, \sigma}^{\ast} \geqslant C_1 \,a_0^{-1}\left( \,a_0^{-1}\left( \sqrt{ \log(2+y)}\right)\right)\right).
 \]
 By definition of stopping times in \eqref{e-tau-n} the left-hand side is the same as
 \[
 \mathbb{P} \left( \sup_{t\in[0,T]} Z_t^{*,x} + W_{0, A, \sigma}^{\ast} > n(y)
 \right),
 \]
 which, by \eqref{eq-a0y}, \eqref{eq-ny}, \eqref{e-a-n-}, \eqref{e-a-n-u} is controlled up to constants by
 \[
 \mathbb{P} \left( \sup_{t\in[0,T]} Z_t^{*,x} + W_{0, A, \sigma}^{\ast} > a_0^{-1}\left(\frac1{\| \sigma^{-1} \|}\left( \frac{\log(y)}{5T}\right)^{1/2}\right)
 \right).
 \]
 Using \eqref{e|Z|-} we can find positive constants $D_1,D_2$ such that this probability is smaller than
 \[
 \mathbb{P} \left( a_0\left( D_1+D_2 W_{0, A, \sigma}^{\ast} \right) \geqslant a_0^{-1}\left(\frac1{\| \sigma^{-1} \|}\left( \frac{\log(y)}{5T}\right)^{1/2}\right)
 \right)=
 \]
 \[
 \mathbb{P}\left( W_{0, A, \sigma}^{\ast} \geqslant \frac{ 1}{ D_2} a_0^{-1}\left(a_0^{-1}\left(\frac1{\| \sigma^{-1} \|}\left( \frac{\log(y)}{5T}\right)^{1/2}\right)
 -\frac{ D_1}{ D_2}
 \right)
 \right).
 \]
 This proves \eqref{e-PSI-d} for a large enough constant $D_2$ and small enough constant $\delta_0$.
\end{proof}

In particular, this estimate applies for the example \begin{equation}
F(t,x)=-x|x|_H^n.
\end{equation} More generally
this estimate applies to any function \begin{equation}
F(t,x)=-xf(t,|x|_H)
\end{equation}
where $f(t,y)\geqslant0$ is a real-valued function such that for any $t\geqslant0,y\geqslant0$ we have $f(t,y)<c(1+y)^{n} $. Note that
$F$ is monotone if for any fixed $t\geqslant0$ function $f(t,y)$ is increasing lower semicontinuous for $y\geqslant0$.
This follows from the sub-gradient representation
\begin{equation}
F(t,x)=-\partial \int _0^{ |x|_H }f(t,s)ds,
\end{equation}
see \cite[\S5 and \S24]{RockafellarBook1970} and \cite[Section 1.2]{BarbuBook2010}.

\begin{rem}
	With more tedious computations, which we do not present in our paper, in more general situations $\PSI$ can be defined as
	\begin{equation}\label{e-PSI-p0-}
	\PSI(y):=\left( \mathbb{P}\left( W_{0, A, \sigma}^{\ast} \geqslant C_1 \,a_0^{-1}\left( C_2 \,a_0^{-1}\left( \sqrt{C_3\log(2+y)}\right)\right)\right) \right)^{\varepsilon_1-1}.
	\end{equation}
for some constants $C_1, C_2,C_3>0$ and $\varepsilon_1>0$.
\end{rem}

\end{ex}

\bibliographystyle{amsplain}

\providecommand{\bysame}{\leavevmode\hbox to3em{\hrulefill}\thinspace}
\providecommand{\MR}{\relax\ifhmode\unskip\space\fi MR }
\providecommand{\MRhref}[2]{%
  \href{http://www.ams.org/mathscinet-getitem?mr=#1}{#2}
}
\providecommand{\href}[2]{#2}

\end{document}